\documentclass[]{amsart}
\usepackage{amsmath}
\usepackage{amssymb,amsfonts}
\usepackage{bbm}
\usepackage{xy}
\xyoption{all}

\newtheorem{theorem}{Theorem}[section]
\newtheorem{definition}[theorem]{Definition}
\newtheorem{proposition}[theorem]{Proposition}
\newtheorem{corollary}[theorem]{Corollary}
\newtheorem{lemma}[theorem]{Lemma}

\numberwithin{equation}{section}

\newcommand{\A}{\mathcal{A}}
\newcommand{\F}{\mathcal{F}}

\newcommand{\vv}{\mathbf{v}}
\newcommand{\ww}{\mathbf{w}}

\newcommand{\B}{\mathcal{B}}
\newcommand{\T}{\mathrm{T}}
\newcommand{\s}{\mathrm{S}}
\newcommand{\HH}{\mathbb{D}}
\newcommand{\Z}{\mathbb{Z}}
\newcommand{\N}{\mathbb{N}}

\begin{document}
\title[(Co)Homology of cyclic monoids]{Computability of the (co)homology of cyclic monoids}
\author{M. Calvo-Cervera}
\author{A.M. Cegarra}
\address{Departamento de \'Algebra,  Universidad de
Granada, 18071 Granada, Spain}
\email{acegarra@ugr.es}
\email{mariacc@ugr.es}
\thanks{This work has been supported by DGI
of Spain, Project MTM2011-22554. Also, the first author by FPU
grant FPU12-01112.}

 \subjclass[2000]{20M50}
 \keywords{monoid, homology, cohomology, free resolution.}

\begin{abstract}
Leech's (co)homology groups of finite cyclic monoids are computed.
\end{abstract}

\maketitle

\section{Introduction} The so-called Leech (co)homology groups of a monoid $M$ are defined to be those of its category of factorizations $\HH \!M$, whose object set is $M$ and arrow set is $M\times M\times M$ (see Section \ref{prel} for details).  Therefore, if $\A:\HH M\to \mathbf{Ab}$ is any left $\HH M$-module, that is, any abelian group valued functor on $\HH M$, the cohomology groups of $M$ with coefficients in $\A$ \cite[Chapter II, \S 2.1]{leech} are defined by
\begin{equation}\label{chn}
\mathrm{H}^n(M,\A)=\mathrm{Ext}_{\HH \!M}^n(\Z,\A)= R^n\mathrm{Hom}_{\HH M}(\Z,-)(\A)=
R^n\mathrm{Hom}_{\HH M}(-,\A)(\Z), \hspace{0.4cm}(n\geq 0),
\end{equation}
where, for any two left $\HH M$-modules $\A$ and $\A'$,  $\mathrm{Hom}_{\HH M}(\A,\A')$ denotes  the abelian group of morphisms of $\HH M$-modules between them, and $\Z:\HH M\to \mathbf{Ab}$ is the constant functor defined by the abelian group of integers $\Z$. Similarly, for $\B:\HH M^{op}\to \mathbf{Ab}$ any right $\HH M$-module, the homology groups of $M$ with coefficients in  $\B$ \cite[Definition 2.1]{Khusi} are defined by
\begin{equation}\label{hn}
\mathrm{H}_n(M,\B)=\mathrm{Tor}^{\HH M}_n(\B,\Z)= L_n(-\otimes_{\HH M}\Z)(\B)=L_n(\B\otimes_{\HH M}-)(\Z), \hspace{0.4cm}(n\geq 0),
\end{equation}
where, for any left $\HH M$-module $\A$, the tensor product $\B\otimes_{\HH M}\A$ is the abelian group defined as the coend of the bifunctor $\HH M^{op}\times \HH M\to \mathbf{Ab}$ which carries each pair $(x,y)\in M\times M$ to the tensor product abelian group $\B(x)\otimes \A(y)$.

It is remarkable that the category of left $\HH M$-modules is equivalent to the category of internal abelian group objects in the comma category of monoids over $M$ and,   with a dimension shift, the Barr-Beck cotriple (co)homology groups \cite{Barr-Beck} and the Leech (co)homology groups of $M$ are naturally isomorphic \cite[Theorem 8]{wells}. Furthermore, when coefficients are taken in ordinary $M$-modules (regarded as constant on objects $\HH M$-modules), Leech (co)-homology groups agree with those by Eilenberg and Mac Lane \cite[Chapter X, \S 5]{maclane}, see Subsection \ref{cohoM-M} below for some details. In particular, Eilenberg-Mac Lane (co)homology groups of groups are instances of Leech (co)homology groups of monoids.

This paper deals with the (co)homology of finite cyclic monoids $C_{m,q}$, whose structure and classification by means of the index $m$ and the period $q$ was first stated by Frobenius \cite{Frobenius}. Although the (co)homology groups of any finite cyclic group $C_q=C_{0,q}$ have been well-known since they were computed in 1949 by Eilenberg \cite[\S 11]{Eilenberg}, this is not the case for finite cyclic monoids of index $m\geq 1$. Indeed, to our knowledge, the cohomology groups $\mathrm{H}^n(C_{m,q},\A)$ of a cyclic monoid of index $m\geq 1$  have been computed  only for $n\leq 2$ by Leech in \cite[Chapter II, 7.20, 7.21]{leech}. However, because higher cohomology groups are interesting (see \cite{cch}, for instance), the aim of this paper is to compute all the (co)homology groups of any finite cyclic monoid $C_{m,q}$.

Briefly, the contents of the paper are as follows.  In Section 2, while fixing notation and terminology, we review some basic constructions concerning the (co)homology of monoids. Section 3 is mainly dedicated to studying the {\em trace maps} associated to any $\HH C_{m,q}$-module, which are a key tool in our deliberations. Section 4 is devoted to the construction of a specific free resolution of the trivial $\HH C_{m,q}$-module $\Z$, which allows us to determine, in the final Section 5, the groups $\mathrm{H}^n(C_{m,q},\A)$ and $\mathrm{H}_n(C_{m,q},\B)$. The (co)homology of $C_{m,q}$ is proven to be periodic  with a period of $2p/\gcd(e,p)$.

\section{Notations and preliminaries.}\label{prel}

Throughout,  monoids $M$ (although not necessarily commutative) are written additively.

\subsection{The category $\HH M$.} As noted in the introduction,  the (co)homology groups of a monoid $M$ are defined to be those of a small category, denoted by $\HH M$, canonically  associated to the monoid. This is the category whose set of objects is $M$ and set of arrows $M\times M\times M$, where $(x,y,z):y\to x+y+z$. Composition is given by $$(u,x+y+z,v)(x,y,z)=(u+x,y,z+v),$$ and the identity morphism of any object $x$ is $(0,x,0):x\to x$.

\subsection{Left $\HH M$-modules.} A left $\HH M$-module is an  abelian group valued functor on the category $\HH M$. If, for any left $\HH M$-module $\A:\HH M\to \mathbf{Ab}$,  we write $$\A(x,y,z)=x_*z^*:\A(y)\to\A(x+y+z),$$ then we see that $\A$ consists of abelian groups $\A(x)$, one for each $x\in M$, and homomorphisms
$$ \A(y)\overset{x_*}\longrightarrow \A(x+y)\overset{y^*}\longleftarrow \A(x),$$
 for each $x,y\in M$, such that the equations below hold.
$$\begin{array}{cc}x_*y_*=(x+y)_*: \A(z)\to \A(x+y+z),& y^*x^*=(x+y)^*: \A(z)\to \A(z+x+y),\\
0_*=0^*=id_{\A(x)}:\A(x)\to \A(x),&
x_*y^*=y^*x_*:\A(z)\to \A(x+z+y).
\end{array}$$

For instance, let $\Z:\HH M\to \mathbf{Ab}$ be the  $\HH M$-module
 that associates to each element $x\in M$ the free abelian
group on the generator $(x)$, $\Z(x)$, and to each  $x,y\in M$ the isomorphisms of abelian groups
$$ \Z(y)\overset{x_*}\longrightarrow \Z(x+y)\overset{y^*}\longleftarrow \Z(x)$$
given on generators by $x_*(y)=(x+y)=y^*(x)$. This is
isomorphic to the  $\HH M$-module defined by the constant functor
on $\HH M$ which associates the abelian group $\Z$
to any $x\in M$.

For two left $\HH M$-modules $\A$ and $\A'$, a morphism between them (i.e., a natural transformation) $f:\A\to \B$ consists of homomorphisms $f_x:\A(x)\to \A'(x)$, such that, for any $x,y\in M$, the squares below commute.
$$
\xymatrix{\A(y)\ar[r]^-{x_*}\ar[d]_{f_y}&\A(x+y)\ar[d]^{f_{x+y}}&\A(x)\ar[l]_-{y^*}
\ar[d]^{f_x}\\ \A'(x)\ar[r]^-{x_*}&\A'(x+y)&\A'(x)\ar[l]_-{y^*}}
$$

The category of left $\HH M$-modules, denoted by $\HH M$-Mod, is an abelian category with enough projective and injective objects. We refer to \cite[Chapter  I, \S 1]{leech} for details, but recall that the set of morphisms between two  $\HH M$-modules $\A$ and $\A'$, denoted by $\mathrm{Hom}_{\HH M}(\A,\A')$, is an abelian group by pointwise addition, that is, if $f,g:\A\to \A'$ are morphisms, then  $f+g:\A\to \A'$ is defined by setting $(f+g)_x=f_x+g_x$, for each $x\in M$. The zero  $\HH M$-module is the constant functor
$0:\HH M\to \mathbf{Ab}$ defined by the trivial abelian group $0$, and a sequence of $\HH M$-modules $\A\to \A'\to \A''$ is exact if and only if the induced sequences of abelian groups $\A(x)\to \A'(x)\to \A''(x)$ are exact, for all $x\in M$.

\subsection{Free left $\HH M$-modules.} Let $\mathbf{Set}\!\!\downarrow\!\!_M$ be
the comma category of sets over
 the underlying set of $M$; that is, the category whose objects
 $S=(S,\pi)$ are sets $S$ endowed with a map $\pi:S\to M$ and
 whose morphisms are maps $\varphi:S\to T$ such that $\pi\varphi=\pi$. There is a {\em forgetful functor}
 $U:\HH M\text{-Mod}
 \to \mathbf{Set}\!\!\downarrow\!\!_M$,
 which carries any $\HH M$-module $\A$ to the disjoint union set
 $$U\A=\coprod_{x\in M}\A(x)=\{(x,a)\mid \,  x\in M,\, a\in \A(x)\},$$
endowed with the projection map $\pi:U\A\to M$, $\pi(x,a)=x$. A morphism  $f:\A\to \A'$ is sent to the map $Uf:U\A\to U\A'$ given by $Uf(x,a)=(x,f_x(a))$.
 There is also   a {\em free left $\HH M$-module}
functor
$ F:\mathbf{Set}\!\!\downarrow\!\!_M \to
\HH M\text{-Mod}$,
 which is defined as follows:
If $S=(S,\pi)$ is any set over $M$, then $ F S$ is the $\HH M$-module such that, for each  $x\in M$,
$$ F S(x)=\Z\{(u,s,v)\in M\times S\times M \mid u+\pi s+v=x\}
$$
is the free abelian group with generators all triples $(u,s,v)$, where
$u,v\in M$ and $s\in S$,  such that $u+\pi s+v=x$.
We usually write $(0,s,0)$ simply by $s$, so that each
element of $s\in S$ is regarded as an element $s\in  F S(\pi s)$.
For any
$x,y\in M$,  the homomorphisms
$$  F S(y)\overset{x_*}\longrightarrow  F S(x+y)\overset{y^*}\longleftarrow  F S(x)$$
are defined on generators by $x_*(u,s,v)=(x+u,s,v)$ and $y^*(u,s,v)=(x,s,v+y)$.

If $\varphi:S\to T$ is any map of sets over $M$, the induced
morphism  $ F \varphi: F S\to  F S'$ is given, at
each $x\in M$, by the homomorphism such that $( F\varphi)_x(u,s,v)=(u,\varphi(s),v)$.

\begin{proposition}\label{adfu} The functor $ F$  is left adjoint to the functor
 $U$. Thus, for $S=(S,\pi)$ any set over $M$ and any left $\HH M$-module $\A$, there is a natural isomorphism of abelian groups
$$\xymatrix{\mathrm{Hom}_{\HH M}( F S,\A) \ \cong \ \prod\limits_{s\in S}\A(\pi s).}$$
\end{proposition}
\begin{proof} At any set $S$ over $M$, the unit of the
adjunction is the map
$$\epsilon:S\to U F S=
\{(x,a)\mid x\in M,\, a\in  F S(x)\},
\hspace{0.5cm} s\mapsto (\pi s, s).
$$
If $\A$ is a $\HH M$-module and $\varphi:S\to U\A$ is any map over
$M$, then the unique morphism of $\HH M$-modules $f: F S\to \A$
such that $( U f)\,\epsilon =\varphi$ is determined by the equations
$$f_x(u,s,v)=u_*v^*\varphi(s),$$ for any $x\in M$ and  $(u,s,v)\in M\times S\times M$ with $u+\pi s+v=x$. Since giving a map over $M$, $\varphi:S\to U\A$, is the same thing as giving a list
$(\varphi(s))_{s\in S}\in \prod_{s\in S}\A(\pi s)$, the isomorphism $\mathrm{Hom}_{\HH M}( F S,\A)  \cong  \prod_{s\in S}\A(\pi s)$ follows.
\end{proof}

\subsection{Right $\HH M$-modules.}
The category  of {\em right $\HH M$-modules} is defined to be the category of functors $\B:{\HH M}^{op}\to \mathbf{Ab}$. A right ${\HH M}$-module $\B$ provides us with abelian groups $\B(x)$, $x\in M$,  and homomorphisms
$$ \B(y)\overset{x^*}\longleftarrow \B(x+y)\overset{y_*}\longrightarrow \B(x),$$
 for each $x,y\in M$, such that the equations below hold.
$$\begin{array}{cc}y^*x^*=(x+y)^*: \B(x+y+z)\to \B(z),& x_*y_*=(x+y)_*: \B(z+x+y)\to \B(z),\\
0_*=0^*=id_{\B(x)}:\B(x)\to \B(x),&
x^*y_*=y_*x^*: \B(x+z+y)\to\B(z).
\end{array}$$

\subsection{Tensor product of $\HH M$-modules.}
If $\B$ is a right $\HH M$-module and  $\A$ is any left $\HH M$-module,  their  tensor product
$\xymatrix{\B\otimes_{\HH M}\A =  \overset{_{\HH M}}\int \B\otimes\A}$
is the abelian group coend \cite[Chapter IX, \S 6]{maclane2} of the functor $\B\otimes\A:\HH M^{op}\times \HH M\to \mathbf{Ab}$ defined by $(\B\otimes\A)(x,y)= \B(x)\otimes \A(y)$.

That is,
$\B\otimes_{\HH M}\A$ is the abelian group generated by elements of the form $b\otimes a$, where $b\in \B(x)$ and $a\in \A(x)$, $ x\in M$, subject to the relations
$$
\begin{array}{cll}
(b+b')\otimes a= b\otimes a +b'\otimes a, &\text{ for} &b,b'\in \B(x),\,a\in \A(x), x\in M,\\
b\otimes (a+a')=b\otimes a+b\otimes a', &\text{ for} &b\in \B(x),\,a,a'\in \A(x), x\in M,\\
 y_*b\otimes a=b\otimes y^*a, &\text{ for}&b\in \B(x+y),\,a\in \A(x),\, x,y\in M,\\
  y^*b\otimes a=b\otimes y_*a, &\text{ for}&b\in \B(y+x),\,a\in \A(x),\, x,y\in M.
\end{array}
$$
\begin{proposition}\label{pwf} For $S=(S,\pi)$ any set over $M$ and $\B$ any right
 $\HH M$-module, there is a natural isomorphism of abelian groups
$$\xymatrix{\B\otimes_{\HH M} F S \ \cong \ \bigoplus\limits_{s\in S}\B(\pi s).}$$
\end{proposition}
\begin{proof}
As an abelian group, $\B\otimes_{\HH M} F S$ is generated by the elements $$b\otimes (u,s,v)=b\otimes u_*v^*(0, s,0)=b\otimes u_*v^*s=u^*v_* b\otimes s,$$ with $u,v\in M$, $s\in S$, and $b\in \B(u+\pi s+v)$. The claimed isomorphism carries such a generator $b\otimes (u,s,v)$ to the element $u^*v_*b\in \B(\pi s)$. Its inverse map carries any element $b\in \B(\pi s)$ to the generator $b\otimes s$ of $\B\otimes_{\HH M} F S$.
\end{proof}

\subsection{Computing the (co)homology of a monoid.} From Proposition \ref{adfu},  it easily follows that every free left $\HH M$-module is projective. Then,  if
$$
\F_\bullet\overset{\epsilon}\to \Z: \hspace{0.5cm}\cdots \to \F _2\overset{\partial}\to \F _1\overset{\partial}\to \F _0\overset{\epsilon}\to \Z\to 0
$$
is any free resolution of $\Z$ in the category of left $\HH M$-modules, then the cohomology groups of $M$  with coefficients in a left $\HH M$-module $\A$, defined in \eqref{chn}, can be computed by means of the induced cochain complex of abelian groups
$$\mathrm{Hom}_{\HH M}(\F _\bullet,\A):  0\to \mathrm{Hom}_{\HH M}(\F _0,\A)\overset{\partial^*}
\to \mathrm{Hom}_{\HH M}(\F _1,\A) \overset{\partial^*}\to\mathrm{Hom}_{\HH M}(\F _2,\A)  \to \cdots $$ by
$
\mathrm{H}^n(M,\A)=\mathrm{H}^n\big(\mathrm{Hom}_{\HH M}(\F _\bullet,\A)\big),
$
and the homology groups of $M$  with coefficients in a right $\HH M$-module $\B$, defined in \eqref{hn}, by means of the induced chain complex
$$ \B\otimes_{\HH M}\F _\bullet:\  \cdots \to \B\otimes_{\HH M}\F _2\overset{id\otimes \partial}\longrightarrow \B\otimes_{\HH M}\F _1\overset{id\otimes \partial}\longrightarrow \B\otimes_{\HH M}\F _0\to 0$$ as
$
\mathrm{H}^n(M,\B)=\mathrm{H}_n\big(\B\otimes_{\HH M}\F _\bullet\big)
$.

\subsection{Eilenberg-Mac Lane (co)homology.}
\label{cohoM-M}  There is a full exact embedding from the category of ordinary left $M$-modules into the category of left $\HH M$-modules. This carries any left $M$-module $A$, with $M$-action $(x,a)\mapsto x a$, to the left $\HH M$-module, also denoted by $A$,  defined by $A(x)=A$, for all $x\in M$, together with the homomorphisms $x_*, x^*:A\to A$  given by $x_*a=x a$ and $x^*a=a$ \cite[Chapter III, Lemma 1.9]{leech}. Similarly, there is a full exact embedding from the category of ordinary right $M$-modules into the category of right $\HH M$-modules, which carries any right $M$-module $B$ to the right $\HH M$-module, also denoted by $B$,  defined by $B(x)=B$, for all $x\in M$, together the homomorphisms $x_*, x^*:B\to B$  given by $x^*b=b x$ and $x_*b=b$.

When, for $A$ any left $M$-module and $B$ any right $M$-module,  one applies the functors $\mathrm{Hom}_{\HH M}(-,A)$ and  $B\otimes_{\HH M}-$ to the standard free resolution of the left $\HH M$-module $\Z$ in \cite[Chaper II, 2.2]{leech}, then one obtains a cochain complex isomorphic to $\mathrm{Hom}_{\Z\!M}(\mathrm{B}(M),A)$ and a chain complex isomorphic to $B\otimes_{\Z M}\mathrm{B}(M)$, respectively. Here, $\Z M$ is the monoid ring and $\mathrm{B}(M)$ the bar resolution of $\Z$ as a left $M$-module. It follows that
$$\begin{array}{l}\mathrm{H}^n(M,A)=\mathrm{Ext}_{\Z M}^n(\Z,A),\hspace{0.5cm}
\mathrm{H}_n(M,B)=\mathrm{Tor}^{\Z M}_n(B,\Z).
\end{array}$$
That is, the Leech (co)homology groups $\mathrm{H}^n(M,A)$ and $\mathrm{H}_n(M,B)$ agree with those by Eilenberg and Mac Lane \cite[Chapter X, \S 5]{maclane} (see the proof of  \cite[Chapter III, Corollary 1.15]{leech} for more details). In particular, Eilenberg-Mac Lane (co)homology groups of groups are instances of Leech (co)homology groups of monoids \footnote{If $G$ is a group, regarded as a category with only one object, then $G$ and $\HH\!G$ become equivalent categories due to the functor $F:G\to \HH\!G$ given by $F(x)=(x,0,-x):0\to 0$. Consequently, the categories of  $G$-modules and of $\HH\!G$-modules are equivalent. This gives an alternative and easier proof that, for groups, both the Leech and the Eilenberg-Maclane (co)homology theories are equivalent.}.

\section{Cyclic monoids and trace maps.}

The structure of finite cyclic monoids was first stated by Frobenius \cite{Frobenius}. Briefly, let us recall that, if $\sim$ is any non-equality congruence on the additive monoid  of natural numbers, $\N=\{0,1,\dots\}$, then the least $m\geq 0$ such that $m\sim x$ for some $x\neq m$ is called the {\em index} of the congruence, and the least $q\geq 1$ such that $m\sim m+q$ is called its {\em period}. Hence,
$$
x\sim y \text{\ \  if and only if either \ } x=y<m, \text{\ or\ } x,y\geq m \text{\ and\  } x\equiv y \!\!\! \mod{q}.
$$
The quotient $
\N/\!\!\sim
$
is called the {\em cyclic monoid of index $m$ and period $q$}, and is denoted here by $C_{m,q}$. As $\N$ is a free monoid on the generator 1, every finite cyclic monoid is isomorphic to a proper quotient of $\N$ and, therefore, to a monoid  $C_{m,q}$ for some $m\geq 0$ and $q\geq 1$.

Since every element of $C_{m,q}$ can be written uniquely in the form $[x]$ with $0\leq x<m+q$, the underlying set of this monoid can be described as the set
$$
C_{m,q}=\{0,1,\dots, m,m+1,\dots, m+q-1\}.
$$
{\em Hereafter, we use this description}. In these terms, the projection map $\wp:\N\to C_{m,q}$ is given by
$$
\wp( x) =\left\{\begin{array}{lll}x&\text{if}&  x<m+q\\[4pt]
x-kp&\text{if}& m+kp\leq x<m+(k+1)p,
\end{array}\right.
$$
and the addition in $C_{m,q}$, which is denoted by the symbol $\oplus$ to avoid confusion with the addition $+$ of $\N$, is given by
$$
x\oplus y =\wp(x+y).
$$
Furthermore, we use the notation
$r\cdot x$, for any $r\in \N$ and $x\in C_{m,q}$, to denote  the element of $C_{m,q}$ defined recursively by
\begin{equation}\label{not2}
0\cdot x=0,\hspace{0.4cm} (r+1)\cdot x= (r\cdot x)\oplus x.
\end{equation}
In other words, $r\cdot x=\overset{\text{\scriptsize{($r$-times})}}{x\oplus \cdots \oplus x}=\wp(\overset{\text{\scriptsize{($r$-times})}}{x+ \cdots + x})=\wp(rx)$. For instance, $2\cdot 8=7$ in $C_{2,9}$.

In what follows, we assume that $m+q\geq 2$, so that $C_{m,q}$ is not the zero monoid.

\vspace{0.2cm}
The following two families of homomorphisms are crucial for our deliberations.
\begin{definition}\label{defst} Let $\A$ be a left $\HH C_{m,q} $-module. For each $x\in C_{m,q}$, $x\geq 1$, the {\em `trace map'}
\begin{equation}\label{trace}
\T:\A(x)\longrightarrow\A(m\oplus (x-1))
\end{equation}
is the homomorphism defined  by
$$
\T(a)=\sum_{t=0}^{m+q-1}t^*(m+q-t-1)_*a-
    \sum_{s=0}^{m-1}s^*(m-s-1)_*a.
$$

Also, for each $x\in C_{m,q}$, let
\begin{equation}\label{s}
\s: \A(x)\to \A(x\oplus 1)
\end{equation}
be the homomorphism defined by
$\hspace{0.2cm}\s(a)=1_*a-1^*a$.
\end{definition}

The following subgroups will be used later.
\begin{equation}\label{not1}
\begin{array}{ll}
\A^\T(x)=\{a\in \A(x)\mid \T(a)=0\},&
\A_\T(x)=\{\T(a)\mid a\in \A(x)\},\\ [5pt]
\A^\s(x)=\{a\in \A(x)\mid \s(a)=0\},&
\A_S(x)=\{\s(a)\mid a\in \A(x)\}.
\end{array}
\end{equation}

\begin{lemma}\label{tsq} For any left $\HH C_{m,q} $-module $\A$, the squares below commute.
$$\xymatrix@C=20pt{
\A(x)\ar[d]_{1^*}\ar[r]^{1_*}&\A(x\oplus 1)\ar[d]^{\T}\\
\A(x\oplus 1)\ar[r]^{\T}&\A(m\oplus x)
} \ \xymatrix@C=20pt{
\A(x)\ar[d]_{1^*}\ar[r]^-{\T}&\A(m\oplus (x-1))\ar[d]^{1^*}\\
\A(x\oplus 1)\ar[r]^{\T}&\A(m\oplus x)
}\xymatrix@C=20pt{
\A(x)\ar[d]_{1_*}\ar[r]^-{\T}&\A(m\oplus (x-1))\ar[d]^{1_*}\\
\A(x\oplus 1)\ar[r]^{\T}&\A(m\oplus x)
}
$$
\end{lemma}
\begin{proof} To prove that $\T\, 1_*=\T\,1^*$, let $a\in \A(x)$. On one hand,
\begin{align*}T(1_*a)&= \sum_{t=0}^{m+q-1}t^*(m+q-t-1)_*1_*a-
    \sum_{s=0}^{m-1}s^*(m-s-1)_*1_*a
    \\
    &= \sum_{t=0}^{m+q-1}t^*\big((m+q-t-1)\oplus 1\big)_*a-
    \sum_{s=0}^{m-1}s^*((m-s-1)\oplus 1)_*a
    \\
&=m_*a+  \sum_{t=1}^{m+q-1}t^*(m+q-t)_*a
-m_*a
- \sum_{s=1}^{m-1}s^*(m-s)_*a\\ &=
\sum_{t=1}^{m+q-1}t^*(m+q-t)_*a-\sum_{s=1}^{m-1}s^*(m-s)_*a,
\end{align*}
and, on the other hand,
\begin{align*}\T(1^*a)&=
\sum_{t=0}^{m+q-1}1^*t^*(m+q-t-1)_*a-
    \sum_{s=0}^{m-1}1^*s^*(m-s-1)_*a\\
    &=\sum_{t=0}^{m+q-1}(1\oplus t)^*(m+q-t-1)_*a-
    \sum_{s=0}^{m-1}(1\oplus s)^*(m-s-1)_*a\\
    &=\sum_{t=0}^{m+q-2}(1+t)^*(m+q-t-1)_*a+m^*a
    -\sum_{s=0}^{m-2}(1+s)^*(m-s-1)_*a-m^*a\\ &=
    \sum_{t=0}^{m+q-2}(1+t)^*(m+q-t-1)_*a-
    \sum_{s=0}^{m-2}(1+s)^*(m-s-1)_*a,
\end{align*}
whence, by comparison, the result follows.

The other two equalities, $1^*\,\T=\T\, 1^*$ and $1_*\,\T=\T\, 1_*$, follow easily from the commutativity of the monoid $C_{m,q}$. \end{proof}

\begin{lemma}\label{stts} For any left $\HH C_{m,q} $-module $\A$, the sequences
$$
\A(x)\overset{\s}\longrightarrow \A(x\oplus 1)\overset{\T}\longrightarrow
\A(m\oplus x), \hspace{0.4cm}
\A(x)\overset{\T}\longrightarrow \A(m\oplus(x-1))\overset{\s}\longrightarrow
\A(m\oplus x),
$$
are semiexact, that is, $\T\,\s=0$ and $\s\,\T=0$.
\end{lemma}
\begin{proof}
By Lemma \ref{tsq},   $\T\s=\T\,1_*-\T\,1^*=0$, and
$\s\T=1_*\,\T-1^*\,\T=\T\,1_*-\T\,1^*=0$.
\end{proof}
\section{A resolution of $\Z$ by free $\HH C_{m,q} $-modules.} It is possible to calculate the (co)homology of cyclic monoids efficiently by a clever choice of resolution. We construct here a specific free resolution of the trivial $\HH C_{m,q} $-module $\Z$, \begin{equation}\label{fres}
\F_\bullet\overset{\epsilon}\to \Z: \hspace{0.5cm}\cdots \to \F_2\overset{\partial}\to \F_1\overset{\partial}\to \F_0\overset{\epsilon}\to \Z\to 0,
\end{equation}
as follows.

For each integer $r\geq 0$, choose symbols $\vv_r$ and $\ww_r$. Then, recalling the notation \eqref{not2},

\begin{itemize}
\item[-] {\em  $\F_{2r}$ is the free $\HH C_{m,q} $-module on the unitary set over $C_{m,q}$, $\{\vv_r\}\overset{\pi}\to C_{m,q}$, where $\pi \vv_r =r\cdot m$.}

\vspace{0.2cm}
\item[-] {\em  $\F_{2r+1}$ is the free $\HH C_{m,q} $-module on the unitary set over $C_{m,q}$, $\{\ww_r\}\overset{\pi}\to C_{m,q}$, where $\pi \ww_r=r\cdot m\oplus 1$.}

\vspace{0.2cm}
\item[-] {\em  The augmentation $\epsilon:\F_0\to \Z$ is the morphism of $\HH C_{m,q} $-modules determined by $$\label{aug}\epsilon_0(\vv_0)=(0)\in \Z(0).$$
    }
    \end{itemize}
\begin{itemize}
\item[-] {\em For each $r\geq 0$, the differential $\partial:\F_{2r+2}\to \F_{2r+1}$ is the morphism of $\HH C_{m,q} $-modules determined by $$ \partial_{\pi\vv_{r+1}}(\vv_{r+1})=\T(\ww_{r}),$$
where $\T:\F_{2r+1}(r\cdot m\oplus 1)\to \F_{2r+1}((r+1)\cdot m)$ is the trace map \eqref{trace}.
}
\end{itemize}

    \begin{itemize}
\item[-] {\em For each $r\geq 0$, the differential $\partial:\F_{2r+1}\to \F_{2r}$ is the morphism of $\HH C_{m,q} $-modules determined by $$ \partial_{\pi\ww_r}(\ww_r)=\s(\vv_r),$$
where $\s:\F_{2r}(r\cdot m)\to \F_{2r}(r\cdot m\oplus 1)$ is the homomorphism \eqref{s}.
}
\end{itemize}

\begin{proposition} $\F_\bullet\overset{\epsilon}\to \Z$, defined  as above, is an augmented  complex of $\HH C_{m,q} $-modules.
\end{proposition}
\begin{proof} The sequence $\F_1\overset{\partial}\to\F_0\overset{\epsilon}\to \Z$ is semiexact, that is, $\epsilon \partial=0$, since
$$
\epsilon_1\partial_1(\ww_0)=\epsilon_1(1_*\vv_0-1^*\vv_0)=
1_*\epsilon_0(\vv_0)-1^*\epsilon_0(\vv_0)=1_*(0)-1^*(0)=(1)-(1)=0.
$$

For any $r\geq 1$, the sequence $\F_{2r+1}\overset{\partial}\to\F_{2r}\overset{\partial}\to \F_{2r-1}$ is semiexact, since
\begin{align*}
\partial_{\pi\ww_r}\partial_{\pi\ww_r}(\ww_r)&=
\partial_{\pi\ww_r}(1_*\vv_r-1^*\vv_r)=1_*\partial_{\pi\vv_r}(\vv_r)
-1^*\partial_{\pi\vv_r}(\vv_r)\\
&=1_*\T(\vv_r)-1^*\T(\vv_r)=\s\T(\vv_r)=0.
\end{align*}

Finally, for any $r\geq 0$, the sequence $\F_{2r+2}\overset{\partial}\to\F_{2r+1}\overset{\partial}\to \F_{2r}$ is also semiexact, since
\begin{align*}
\partial_{\pi\vv_{r+1}}&\partial_{\pi\vv_{r+1}}(\vv_{r+1})=\partial_{\pi\vv_{r+1}}\T(\ww_r)=\\
&
=\partial_{\pi\vv_{r+1}}\Big(\sum_{t=0}^{m+q-1}t^*(m+q-t-1)_*\ww_{r}\Big)-
  \partial_{\pi\vv_{r+1}} \Big( \sum_{s=0}^{m-1}s^*(m-s-1)_*\ww_{r} \Big)
  \\
&
=\sum_{t=0}^{m+q-1}t^*(m+q-t-1)_*\partial_{\pi\ww_{r}}(\ww_{r})-
   \sum_{s=0}^{m-1}s^*(m-s-1)_*\partial_{\pi\ww_{r}}(\ww_{r})
   \\
&
=\sum_{t=0}^{m+q-1}t^*(m+q-t-1)_*\s(\ww_{r})-
   \sum_{s=0}^{m-1}s^*(m-s-1)_*\s(\ww_{r} )=\T\s(\ww_{r})=0.
\end{align*}
\end{proof}

We are now ready to establish the main result of this section.
\begin{theorem}\label{mtheo1} $\F_\bullet\overset{\epsilon}\to \Z$, defined as above,  is a free resolution of the $\HH C_{m,q}$-module $\Z$.
\end{theorem}
\begin{proof} We only have to prove its exactness or, equivalently,  that, for any fixed $x\in C_{m,q}$, the augmented complex of abelian groups
\begin{equation}\label{frex}
\F_\bullet(x)\overset{\epsilon_x}\to \Z(x):
\hspace{0.5cm}\cdots \to \F_2(x)\overset{\partial_x}\to \F_1(x)\overset{\partial_x}\to \F_0(x)\overset{\epsilon_x}\to \Z(x)\to 0,
\end{equation}
is exact. To do so, we are going to show that it has a contracting homotopy. That is, there are homomorphisms $\phi:\Z(x)\to \F_0(x)$ and $\Phi:\F_n(x)\to\F_{n+1}(x)$ for $n\geq 0$, such that $\epsilon_x\phi=id_{\Z(x)}$, $\phi\,\epsilon_x+ \partial_x\Phi=id_{\F_0(x)} $, and for $n\geq 1$, $\Phi\partial_x+\partial_x\Phi=id_{\F_n(x)}$.

These homomorphisms $\phi$ and $\Phi$ are defined on the generators and extended linearly. Recall that $\Z(x)$ is the free abelian group on the generator $(x)$  and, for each  $r\geq 0$, $\F_{2r}(x)$ is the free abelian group on the set
$$\{(u,\vv_r,v)\mid u,v\in C_{m,q} \text{ with }  u\oplus r\cdot m \oplus v=x\},$$
and  $\F_{2r+1}(x)$ is the free abelian group on the set
$$\{(u,\ww_r,v)\mid u,v\in C_{m,q} \text{ with }  u\oplus r\cdot m\oplus v\oplus 1=x\}.$$
Then, we define

\begin{itemize}
\item[-]  $\phi:\Z(x)\to \F_0(x)$ to be the homomorphism determined by $$\phi(x)=(0,\vv_0,x).$$
\end{itemize}
and, for $r\geq 0$,

\begin{itemize}
\item[-]  $\Phi:\F_{2r+1}(x)\to \F_{2r+2}(x)$ to be the homomorphism determined by $$\Phi(u,\ww_r,v)=\left\{\begin{array}{cll}0&\text{ if} & u<m+q-1\\
    (0,\vv_{r+1},v)&\text{ if} & u=m+q-1  \end{array}\right. $$
\end{itemize}
\begin{itemize}
\item[-]  $\Phi:\F_{2r}(x)\to \F_{2r+1}(x)$ to be the homomorphism determined by $$\Phi(u,\vv_r,v)=\sum_{t=0}^{u-1}(t,\ww_r,v\oplus(u-t-1)).$$
\end{itemize}
So defined, we prove that these homomorphisms establish a contracting homotopy on the augmented chain complex \eqref{frex} as follows.

\vspace{0.2cm}

\underline{$\epsilon_x\phi=id_{\Z(x)}$}, since
$$
\epsilon_x\phi(x)=\epsilon_x(0,\vv_0,x)=\epsilon_x(x^*\vv_0)=x^*\epsilon_0(\vv_0)
=x^*(0)=(x).
$$

\underline{$\partial_x\Phi+\phi\,\epsilon_x=id_{\F_0(x)}$}, since, for any $u,v\in C_{m,q}$ with $u\oplus v=x$,
\begin{align*}\partial_x\Phi(u,\vv_0,v)&=\partial_x\Big(\sum_{t=0}^{u-1}(t,\ww_0,
v\oplus(u-t-1))\Big)
=\partial_x\Big(\sum_{t=0}^{u-1}t_*(v\oplus(u-t-1))^*\ww_0\Big)\\
&=\sum_{t=0}^{u-1}t_*(v\oplus(u-t-1))^*\partial_{1}(\ww_0)=
\sum_{t=0}^{u-1}t_*(v\oplus(u-t-1))^*(1_*\vv_0-1^*\vv_0)\\&=
\sum_{t=0}^{u-1}(t+1)_*\wp(u+v-t-1)^*\vv_0-\sum_{t=0}^{u-1}t_*\wp(u+v-t)^*\vv_0\\
&= u_*v^*\vv_0-\wp(u+v)^*\vv_0=u_*v^*\vv_0-x^*\vv_0=(u,\vv_0,v)-(0,\vv_0,x),
\end{align*}
\begin{align*}\phi\,\epsilon_x(u,\vv_0,v)&=\phi\,\epsilon_x(u_*v^*\vv_0)=
\phi(u_*v^*\epsilon_0(\vv_0))=\phi(u_*v^*(0))=\phi(u\oplus v)=\phi(x)\\&=(0,\vv_0,x),
\end{align*}
and therefore $(\partial_x\Phi+\phi\,\epsilon_x)(u,\vv_0,v)=(u,\vv_0,v)$, for any generator $(u,\vv_0,v)$ of $\F_0(x)$.

\vspace{0.2cm}
\underline{$\partial_x\Phi+\Phi\partial_x=id_{\F_{2r+1}}(x)$}, since for any generator $(u,\ww_r,v)$ of $\F_{2r+1}(x)$ with $u<m+q-1$,

\begin{align*}(\partial_x\Phi+\Phi\partial_x)(u,\ww_r,v)&=\Phi\partial_x(u,\ww_r,v)
=\Phi\partial_x(u_*v^*\ww_r)=\Phi\big(u_*v^*\partial_{\pi\ww_r}(\ww_r)\big)\\&=
\Phi\big(u_*v^*(1_*\vv_r-1^*\vv_r)\big)=\Phi\big((u+1)_*v^*\vv_r-u_*(v\oplus 1)^*\vv_r)\big)\\&=
\Phi(u+1,\vv_r,v)-\Phi(u,\vv_r,\wp(v+1))\\&=
\sum_{t=0}^{u}(t,\ww_r,\wp(u+v-t))-\sum_{t=0}^{u-1}(t,\ww_r,\wp(u+v-t))\\&=
(u,\ww_r,\wp(v))=(u,\ww_r,v),
\end{align*}
while for generators $(m+q-1,\ww_r, v)$ of $\F_{2r+1}(x)$, we have
\begin{align*}\partial_x\Phi(m+q-1,\ww_r,v)&=\partial_x(0,\vv_{r+1},v)=\partial_x(v^*\vv_{r+1})=
v^*\partial_{\pi \vv_{r+1}}(\vv_{r+1})\\&=\sum_{t=0}^{m+q-1}v^*t_*(m+q-t-1)^*\ww_r-
\sum_{t=0}^{m-1}v^*t_*(m-t-1)^*\ww_r\\&=
\sum_{t=0}^{m+q-1}(t,\ww_r,\wp(v+m+q-t-1))-\sum_{t=0}^{m-1}(t,\ww_r,\wp(v+m-t-1))\\&=
\sum_{t=0}^{m+q-1}(t,\ww_r,\wp(v+m+q-t-1))-\Phi(m,\vv_r,v),
\end{align*}
\begin{align*}\Phi\partial_x(m+q-1,\ww_r,v)&= \Phi\big((m+q-1)_*v^*\partial_{\pi\ww_r}(\ww_r)\big)=
\Phi\big((m+q-1)_*v^*(1_*\vv_r-1^*\vv_r)\big)\\&=
\Phi(m_*v^*\vv_r)-\Phi\big((m+q-1)_*(1\oplus v)^*\vv_r)\big)\\&=
\Phi(m,\vv_r,v)-\Phi(m+q-1,\vv_r,\wp(1+v))\\&=
\Phi(m,\vv_r,v)-\sum_{t=0}^{m+q-2}(t,\ww_r,\wp(v+m+q-t-1)),
\end{align*}
whence $(\partial_x\Phi+\Phi\partial_x)(m+q-1,\ww_r,v)=(m+q-1,\ww_r,\wp(v))=(m+q-1,\ww_r,v)$.

\vspace{0.2cm}
And, finally, we prove that \underline{$\partial_x\Phi+\Phi\partial_x=id_{\F_{2r}}(x)$}. To do that, let  $(u,\vv_r,v)$ be any fixed generator of $\F_{2r}(x)$. Then, on the one hand,
\begin{align*}\partial_x\Phi(u,\vv_r,v)&=\partial_x\Big(\sum_{t=0}^{u-1}
\big(t,\ww_r,(v\oplus(u-t-1))\big)\Big)
=\partial_x\Big(\sum_{t=0}^{u-1}t_*(v\oplus(u-t-1))^*\ww_r\Big)\\
&=\sum_{t=0}^{u-1}t_*(v\oplus(u-t-1))^*\partial_{\pi\ww_r}(\ww_r)=
\sum_{t=0}^{u-1}t_*(v\oplus(u-t-1))^*(1_*\vv_r-1^*\vv_r)\\ &=
\sum_{t=0}^{u-1}(t+1)_*\wp(u+v-t-1)^*\vv_r-\sum_{t=0}^{u-1}t_*\wp(u+v-t)^*\vv_r\\
&= u_*v^*\vv_r-\wp(u+v)^*\vv_r=u_*v^*\vv_r-(u\oplus v)^*\vv_r=(u,\vv_r,v)-(0,\vv_r,u\oplus v),
\end{align*}
while, on the other hand, we have
\begin{align*}
\Phi\partial_x&(u,\vv_r,v)=\Phi\big(u_*v^*\partial_{\pi\vv_r}(\vv_r)\big)\\  &=
\Phi\Big(u_*v^*\sum_{t=0}^{m+q-1}t_*(m+q-t-1)^*\ww_{r-1}-u_*v^*\sum_{t=0}^{m-1}t_*(m-t-1)^*\ww_{r-1} \Big)\\ &=
\sum_{t=0}^{m+q-1}\Phi\big(u\oplus t,\ww_{r-1},v\oplus (m+q-t-1)\big)-
\sum_{t=0}^{m-1}\Phi\big(u\oplus t,\ww_{r-1},v\oplus (m-t-1)\big).
\end{align*}
 Now, if $l\geq 0$ is integer such that $lq<u\leq (l+1)q$, then it is easy to see that the various $t$, with $0\leq t\leq m+q-1$ (resp. $0\leq t\leq m-1$), such that $u\oplus t=m+q-1$, that is, $\wp(u+t)=u+q-1$,  are just those of the form $t=m+(k+1)q-1-u$ for $0\leq k\leq l$ (resp. $0\leq k\leq l-1$). Hence,
 \begin{align*}
\Phi\partial_x(u,\vv_r,v)&=
\sum_{k=0}^{l}(0,\vv_r,v\oplus(u-kq))-\sum_{k=0}^{l-1}(0,\vv_r,v\oplus(u-(k+1)q))=(0,\vv_r,v\oplus u),
 \end{align*}
and thus we get $(\partial_x\Phi+\Phi\partial_x)(u,\vv_r,v)=(u,\vv_r,v)$. This makes complete the proof. \end{proof}

\section{The (co)homology groups of $C_{m,q}$.}
By Theorem \ref{mtheo1}, the (co)homology groups of the cyclic monoid of index $m$ and period $q$, $C_{m,q}$,  can be computed by means of the complex $\F_\bullet$ in \eqref{fres} as
$$\begin{array}{l}
\mathrm{H}^n(C_{m,q},\A)= \mathrm{H}^n\text{Hom}_{\HH C_{m,q} }(\F_\bullet,\A),\\[6pt]
\mathrm{H}_n(C_{m,q},\B)= \mathrm{H}_n(\B\otimes_{\HH C_{m,q} }\F_\bullet),
\end{array}
$$
for $\A$ any left $\HH C_{m,q} $-module and $\B$ any right $\HH C_{m,q} $-module.

Now, for each $r\geq 0$, the $\HH C_{m,q} $-module $\F_{2r}$ is free on the unitary set $\{\vv_r\}$ with $\pi \vv_r=r\cdot m$, while $\F_{2r+1}$ is free on the unitary set $\{\ww_r\}$ with $\pi \ww_r=r\cdot m\oplus 1$. Then, by Proposition \ref{adfu}, there are natural isomorphisms
$$
\text{Hom}_{\HH C_{\!m,q} }(\F_{2r},\A)\cong \A(r\cdot m), \hspace{0.3cm}
\text{Hom}_{\HH C_{\!m,q} }(\F_{2r+1},\A)\cong \A(r\cdot m\oplus 1),
$$
respectively given by $f\mapsto f_{\pi\vv_r}(\vv_r)$ and $g\mapsto g_{\pi\ww_r}(\ww_r)$, which make the diagram
$$\xymatrix@C=15pt{\text{Hom}_{\HH C_{\!m,q\!} }(\F_{2r},\A)\ar[r]^{\partial^*}
\ar[d]_{\textstyle \cong}&
\text{Hom}_{\HH C_{\!m,q\!} }(\F_{2r+1},\A)\ar[r]^{\partial^*}\ar[d]_{\textstyle \cong}&
\text{Hom}_{\HH C_{\!m,q\!} }(\F_{2r+2},\A)\ar[d]_{\textstyle \cong} \\
\A(r\cdot m)\ar[r]^{\s}&\A(r\cdot m\oplus 1)\ar[r]^{\T}&
\A((r+1)\cdot m),}
$$
commutative, where $\s$ and $\T$ are the homomorphisms \eqref{s} and \eqref{trace} in Definition \ref{defst}.
Therefore, recalling the notations in \eqref{not1}, we obtain:
\begin{theorem}\label{mth2} Let $\A$ be any left $\HH C_{m,q} $-module. Then,
$$\mathrm{H}^0(C_{m,q},\A)\cong \A^\s(0), $$
and, for any $r\geq 0$,
\begin{equation}\label{cal1cy}
\begin{array}{lll} \mathrm{H}^{2r+1}(C_{m,q},\A)\cong \frac{\textstyle  \A^\T(r\cdot m\oplus 1)}{\textstyle \A_\s(r\cdot m )},&&
\mathrm{H}^{2r+2}(C_{m,q},\A)\cong \frac{\textstyle  \A^\s((r+1)\cdot m)}{\textstyle \A_\T(r\cdot m\oplus 1)}.
\end{array}
\end{equation}
\end{theorem}

For instance, let us consider the $\HH C_{m,q} $-module $\Z$ for coefficients. In this case, for any $x\in C_{m,q}$, $x\geq 1$,  the trace map $\T:\Z(x)\to \Z(m\oplus (x-1))$ is the homomorphism of multiplication by $q$, since \begin{align*}\T(x)&=\sum_{i=0}^{m+q-1}t^*(m+q-t-1)_*(x)-\sum_{i=0}^{m-1}t^*(m-t-1)_*(x)\\&=
\sum_{i=0}^{m+q-1}(m\oplus (x-1))-\sum_{i=0}^{m-1}(m\oplus (x-1))=p (m\oplus (x-1)),
\end{align*}
while, for all $x$,  $\s:\Z(x)\to \Z(x\oplus 1)$ is the zero homomorphism, since $$\s(x)=1_*(x)-1^*(x)=(1\oplus x)-(x\oplus 1)=0.$$
Therefore, $\mathrm{H}^{0}(C_{m,q},\mathbb{Z})\cong \Z$ and, for any $r\geq 0$,
$$
\mathrm{H}^{2r+1}(C_{m,q},\mathbb{Z})\cong 0, \hspace{0.3cm}
\mathrm{H}^{2r+2}(C_{m,q},\mathbb{Z})\cong \mathbb{Z}/q\mathbb{Z}.
$$

We should note that the isomorphisms \eqref{cal1cy} in the particular case when $r=0$, that is,
$$
\begin{array}{lll} \mathrm{H}^{1}(C_{m,q},\A)\cong \frac{\textstyle  \A^\T(1)}{\textstyle \A_\s(0)}, &&
\mathrm{H}^{2}(C_{m,q},\A)\cong \frac{\textstyle  \A^\s(m)}{\textstyle \A_\T(1)},
\end{array}
$$
were proven by Leech in \cite[Chapter II, 7.20, 7.21]{leech}.

\vspace{0.2cm}
As for homology, if $\B$ is any right $\HH C_{m,q} $-module, by Proposition \ref{pwf}, there are natural isomorphisms
$$
\B\otimes_{\HH C_{\!m,q} }\F_{2r}\cong \B(r\cdot m), \hspace{0.3cm}
\B\otimes_{\HH C_{\!m,q} }\F_{2r+1}\cong \B(r\cdot m\oplus 1),
$$
respectively given on generators by $a'\otimes (u,\vv_r,v)\mapsto u^*v_*a'$ and $a'\otimes (u,\ww_r,v)\mapsto u^*v_*a'$, which make the diagram
$$\xymatrix{\B\otimes_{\HH C_{\!m,q} }\F_{2r+2}\ar[r]^{id\otimes \partial}
\ar[d]_{\textstyle \cong}&
\B\otimes_{\HH C_{\!m,q} }\F_{2r+1}\ar[r]^{id\otimes \partial}\ar[d]_{\textstyle \cong}&
\B\otimes_{\HH C_{\!m,q} }\F_{2r}\ar[d]_{\textstyle \cong} \\
\B((r+1)\cdot m)\ar[r]^{\T}&\B(r\cdot m\oplus 1)\ar[r]^{\s}&
\B(r\cdot m),}
$$
commutative, where, for each $x\in C_{m,q}$, $x\geq 1$, the homomorphism
$
\T:\B(m\oplus (x-1))\longrightarrow \B(x)
$
is the {\em `trace map'}, defined by
$$
\T(b)=\sum_{t=0}^{m+q-1}t^*(m+q-t-1)_*b-
    \sum_{t=0}^{m-1}t^*(m-t-1)_*b,
$$
and, for any $x\in C_{m,q}$,
$
\s: \B(x\oplus 1) \to \B(x)
$
is the homomorphism defined by
$ \s(b)=1_*b-1^*b$.

Then, introducing the subgroups (parallel to those in \eqref{not1})
$$
\begin{array}{ll}
\B^\T(x)=\{b\in \B(m\oplus (x-1))\mid \T(b)=0\},&
\B_\T(x)=\{\T(b)\mid b\in \B(m\oplus (x-1)\},\\ [5pt]
\B^\s(x)=\{b\in \B(x\oplus 1)\mid \s(b)=0\},&
\B_\s(x)=\{\s(b)\mid b\in \B(x\oplus 1)\},
\end{array}
$$
we have the following.
\begin{theorem} \label{mth3} Let $\B$ be any right $\HH C_{m,q} $-module. Then,
$$\mathrm{H}_0(C_{m,q},\B)\cong \frac{\textstyle \B(0)}{\B_\s(0)}, $$
and, for any $r\geq 0$,
\begin{equation}\label{cal2cy}
\begin{array}{lll} \mathrm{H}_{2r+1}(C_{m,q},\B)\cong \frac{\textstyle  \B^\s(r\cdot m)}{\textstyle \B_\T(r\cdot m\oplus 1)},&&
\mathrm{H}_{2r+2}(C_{m,q},\B)\cong \frac{\textstyle  \B^\T(r\cdot m\oplus 1)}{\textstyle \B_\s((r+1)\cdot m)}.
\end{array}
\end{equation}
\end{theorem}

Thus, for example,
$$
\begin{array}{lll} \mathrm{H}_{1}(C_{m,q},\B)\cong \frac{\textstyle  \B^\s(0)}{\textstyle \B_\T(1)},&&
\mathrm{H}_{2}(C_{m,q},\B)\cong \frac{\textstyle  \B^\T(1)}{\textstyle \B_\s(m)}.
\end{array}
$$

It is well-known that the (co)homology of a finite cyclic group $C_q=C_{0,q}$ is periodic with a period of 2. Indeed, when $m=0$, isomorphisms \eqref{cal1cy} and \eqref{cal2cy} state that, for any left $\HH C_{q}$-module $\A$ and right $\HH C_{q}$-module $\B$, and any integer $r\geq 0$, there are isomorphisms
$$ \mathrm{H}^{2r+1}(C_q ,\A)\cong   \A^\T(1)/ \A_\s(0) ,\hspace{0.4cm}
 \mathrm{H}^{2r+2}(C_q ,\A)\cong  \A^\s(0)/ \A_\T(1),
$$
$$ \mathrm{H}_{2r+1}(C_q ,\B)\cong   \B^\s(0)/ \B_\T(1) ,\hspace{0.4cm}
 \mathrm{H}_{2r+2}(C_q ,\B)\cong  \B^\T(1)/ \B_\s(0),
 $$
whence the periodicity of the (co)homology of $C_q$ follows trivially.
The following proposition states that, from dimension 3 onwards, the (co)homology of any finite cyclic monoid $C_{m,q}$ is periodic with a period of $2q/(m,q)$, where $(m,q)$ denotes the greatest common divisor of the index and the period. More precisely,

\begin{proposition} Let $p, n\geq 3$ be integers such that $p\equiv n\!\!\mod{2q/(m,q)}$. Then, for any left $\HH C_{m,q} $-module $\A$ and right $\HH C_{m,q} $-module $\B$, there are isomorphisms
\begin{equation}\label{isoper}
 \mathrm{H}^{p}(C_{m,q},\A)\cong \mathrm{H}^{n}(C_{m,q},\A),\hspace{0.4cm}
  \mathrm{H}_{p}(C_{m,q},\B)\cong \mathrm{H}_{n}(C_{m,q},\B).
\end{equation}

If $m=1$, then there are also  isomorphisms
$$
 \mathrm{H}^{n}(C_{1,q},\A)\cong \mathrm{H}^{2}(C_{1,q},\A),\hspace{0.4cm}
  \mathrm{H}_{n}(C_{1,q},\B)\cong \mathrm{H}_{2}(C_{1,q},\B),
$$
for any $n\geq 2$ such that $n\equiv 2\!\!\mod{2q}$.

\end{proposition}
\begin{proof} Let $p,n\geq 3$ be integers such that $p\equiv n\!\!\mod{2q/(m,q)}$. Then, $p\equiv n\!\!\mod{2}$ and we can write $p=2r+1$ and $n=2s+1$ or $p=2r+2$ and $n=2s+2$ for some integers $r,s\geq 1$ satisfying $r\equiv s\!\!\mod{q/(m,q)}$ or, equivalently, satisfying that $rm\equiv sm\!\!\mod{q}$. Hence, $r\cdot m=s\cdot m$, $r\cdot m\oplus 1=s\cdot m\oplus 1$, and $(r\cdot m)\oplus m=(s\cdot m)\oplus m$, whence the isomorphisms in \eqref{isoper} follow from those in \eqref{cal1cy} and \eqref{cal2cy}.

Suppose now that the cyclic monoid is of index one\footnote{A cyclic monoid of index $m=1$ and period $q$ is the same thing that a cyclic group of order $q$ with a identity adjoined.}, and let $r\geq 0$ be  such that $r\equiv 0 \!\! \mod{q}$. Then $r\oplus 1=1$, and therefore
$$
\mathrm{H}^{2r+2}(C_{1,q},\A)\cong \frac{\textstyle  \A^\s(r\oplus 1)}{\textstyle \A_\T(r\oplus 1)}=\frac{\textstyle  \A^\s(1)}{\textstyle \A_\T(1)}\cong \mathrm{H}^{2}(C_{1,q},\A),
$$
$$
\mathrm{H}_{2r+2}(C_{1,q},\B)\cong \frac{\textstyle  \B^\T(r\oplus 1)}{\textstyle \B_\s(r\oplus 1)}=\frac{\textstyle  \B^\T(1)}{\textstyle \B_\s(1)}\cong \mathrm{H}_{2}(C_{1,q},\B).
$$
\end{proof}

Our results in Theorems \ref{mth2} and \ref{mth3} specify in a simpler form for (co)homology with coefficients in  $C_{m,q}$-modules (see Subsection \ref{cohoM-M}).

\begin{corollary}$(i)$ Let $A$ be any left $C_{m,q}$-module.  Then, $$\mathrm{H}^0(C_{m,q},A)\cong A^\s, $$
and, for any $r\geq 0$,
$$
\begin{array}{lll} \mathrm{H}^{2r+1}(C_{m,q},A)\cong A^\T/A_\s,&&
\mathrm{H}^{2r+2}(C_{m,q},A)\cong A^\s/A_\T,
\end{array}
$$
where $\s,\T:A\to A$ are the homomorphisms given by
$$
\s(a)=1_*a-a, \hspace{0.4cm} \T(a)=m_*\sum\limits_{t=0}^{q-1}t_*a,
$$
$A^\T=\mathrm{Ker}\T$, $A_\T=\mathrm{Im}\T$, $A^\s=\mathrm{Ker}\s$,  and $A_S=\mathrm{Im}\s$.

\vspace{0.2cm}
$(ii)$ Let $B$ be any right $C_{m,q}$-module.  Then, $$\mathrm{H}_0(C_{m,q},B)\cong B/B_\s, $$
and, for any $r\geq 0$,
$$
\begin{array}{lll} \mathrm{H}_{2r+1}(C_{m,q},B)\cong B^\s/B_\T,&&
\mathrm{H}^{2r+2}(C_{m,q},B)\cong B^\T/B_\s,
\end{array}
$$
where $\s,\T:B\to B$ are the homomorphisms given by
$$
\s(b)=b-1^*b, \hspace{0.4cm} \T(b)=m^*\sum\limits_{t=0}^{q-1}t^*b,
$$
$B^\T=\mathrm{Ker}\T$, $B_\T=\mathrm{Im}\T$, $B^\s=\mathrm{Ker}\s$,  and $B_S=\mathrm{Im}\s$.
\end{corollary}

The isomorphism $\mathrm{H}^{2}(C_{m,q},A)\cong A^\s/A_\T$ is already known, see \cite[Proposition 4.1]{grillet} for a recent proof. As an immediate consequence of the above corollary, we see that the Eilenberg-Mac Lane (co)homology of any $C_{m,q}$ is periodic with a period of 2, that is,

\begin{corollary} Let $A$ be a left $C_{m,q}$-module and let $B$ be a right $C_{m,q}$-module. For any $r\geq 0$, there are natural isomorphisms
$$ \mathrm{H}^{2r+1}(C_{m,q},A)\cong  \mathrm{H}^{1}(C_{m,q},A)  ,\hspace{0.4cm}
 \mathrm{H}^{2r+2}(C_{m,q}, A)\cong \mathrm{H}^{2}(C_{m,q}, A),
$$
$$ \mathrm{H}_{2r+1}(C_{m,q},B)\cong  \mathrm{H}_{1}(C_{m,q},B)  ,\hspace{0.4cm}
 \mathrm{H}_{2r+2}(C_{m,q}, B)\cong \mathrm{H}_{2}(C_{m,q}, B).
$$
\end{corollary}

If $A$ is any abelian group, regarded as a left or right $C_{m,q}$-module on which the monoid acts trivially, then, for  any $a\in A$,  $\s(a)=a-a=0$, that is,  $\s=0:A\to A$ is the zero homomorphism, while $\T(a)=\sum_{i=0}^{q-1}a=q\,a$, that is, the trace map
$\T=q:A\to A$ is  multiplication by $q$. Therefore,
$$\mathrm{H}^{0}(C_{m,q},A)\cong A\cong\mathrm{H}_{0}(C_{m,q},A),$$
and, for all $r\geq 0$,
$$
\begin{array}{clclc}
\mathrm{H}^{2r+1}(C_{m,q},A)&\cong&\mathrm{Ker}(q:A\to A)&\cong& \mathrm{H}_{2r+2}(C_{m,q},A),\\[5pt]
\mathrm{H}^{2r+2}(C_{m,q},A)&\cong&\mathrm{Coker}(q:A\to A)&\cong& \mathrm{H}_{2r+1}(C_{m,q},A).
\end{array}
$$
Observe that the (co)homology groups of the finite cyclic monoid $C_{m,q}$ with coefficients in the abelian group do not depend on the index $m$. Indeed,  they  agree with those of the cyclic group $C_q$. Actually, this fact is not surprising because it is well-known that the (co)homology groups of any commutative monoid with trivial coefficients coincide with those of its group reflection (i.e., its image  under the left adjoint of the forgetful functor from groups to monoids) \cite[Proposition 4.4]{fied}, and the group reflection of $C_{m,q}$ is just $C_q$.

To conclude, we  particularize to the case when the coefficients are {\em symmetric} $\HH C_{m,q} $-modules. Recall that, if $M$ is any {\em commutative} monoid, a left  $\HH M$-module $\A$  is called symmetric if, for any $x,y\in M$, $y_*=y^*:\A(x)\to \A(x+y)$. Symmetric $\HH M$-modules are equivalent to abelian group objects in the comma category of commutative monoids over $M$ \cite[Chap. XXII, §2]{grillet2}, and therefore they are the coefficients for the cotriple cohomology theory \cite{Barr-Beck} of commutative monoids, carefully studied by Grillet to whose book \cite{grillet2} we refer interested readers. See also the recent approach to the (co)homology of commutative monoids by Kurdiani and Pirashvili in \cite{pirash}. Symmetric right $\HH M$-modules are defined similarly, and Theorems \ref{mth2} and \ref{mth3} give the following.

\begin{corollary}
$(i)$ Let $\A$ be any symmetric left $C_{m,q}$-module.  Then, $$\mathrm{H}^0(C_{m,q},\A)\cong \A(0), $$
and, for any $r\geq 0$,
$$
\begin{array}{lll} \mathrm{H}^{2r+1}(C_{m,q},\A)\cong \A^\T(r\cdot m\oplus 1),&&
\mathrm{H}^{2r+2}(C_{m,q},A)\cong \frac{\textstyle \A((r+1)\cdot m)}{\textstyle \A_\T(r\cdot m\oplus 1)},
\end{array}
$$
where, for any $x\in C_{m,q}$, $x\geq 1$,   $\T:\A(x)\to \A(m\oplus (x-1))$ is the trace map given by
$$
\T(a)= (m+q) \big((m+q-1)_*a\big)- m\big((m-1)_*a\big),
$$
$A^\T(x)=\mathrm{Ker}\T$, and $A_\T(x)=\mathrm{Im}\T$.

\vspace{0.2cm}
$(ii)$ Let $\B$ be any symmetric right $C_{m,q}$-module.  Then, $$\mathrm{H}_0(C_{m,q},\B)\cong \B(0), $$
and, for any $r\geq 0$,
$$
\begin{array}{lll} \mathrm{H}_{2r+1}(C_{m,q},\B)\cong \frac{\textstyle \B(r\cdot m)}{\textstyle \B_\T(r\cdot m\oplus 1)},&&
\mathrm{H}^{2r+2}(C_{m,q},\B)\cong \B^\T(r\cdot m\oplus 1),
\end{array}
$$
where, for any $x\in C_{m,q}$, $x\geq 1$,   $\T:\B(m\oplus (x-1))\to \B(x)$ is the trace map given by
$$
\T(b)= (m+q) \big((m+q-1)_*b\big)- m\big((m-1)_*b\big),
$$
$B^\T(x)=\mathrm{Ker}\T$, and $B_\T(x)=\mathrm{Im}\T$.
\end{corollary}

\end{document}